\documentclass[12pt,a4paper,papersize]{article}

\usepackage{amsmath,amssymb,amsthm}

\def\ssm{\smallsetminus}

\def\restrictedto%
{\mathchoice%
    {\!\upharpoonright\!}
    {\!\upharpoonright\!}
    {\upharpoonright}
    {\upharpoonright}
}

\def\st{\mathchoice{:}{:}{\,:\,}{\,:\,}}

\newcommand{\splitnodes}[2][{}]%
	{\operatorname{\textup{\textsf{split}}}_{#1}({#2})}

\def\disju{\mathbin{\dot{\cup}}}

\let\emptyset\varnothing

\theoremstyle{plain}
\newtheorem{thm}{Theorem}[section]

\newtheorem{cor}[thm]{Corollary}
\newtheorem{prop}[thm]{Proposition}

\theoremstyle{definition}
\newtheorem{defn}[thm]{Definition}
\theoremstyle{remark}

\begin{document}

\title{Strategic equivalence among 
	hat puzzles of various protocols 
	with many colors}
\author{Masaru Kada \and Souji Shizuma}
\date{September 17, 2019}

\maketitle

\renewcommand{\thefootnote}{\fnsymbol{footnote}}
\renewcommand{\thefootnote}{\arabic{footnote}}

\begin{abstract}
We discuss 
``puzzles of prisoners and hats'' 
with infinitely many prisoners 
and more than two hat colors. 
Assuming that the set of hat colors 
is equipped with a commutative group structure, 
we 
prove strategic equivalence 
among puzzles of several protocols 
with countably many prisoners. 
\end{abstract}

\section{Introduction}\label{sec:intro}

The famous mathematical puzzles 
known as ``puzzles of prisoners and hats'' 
or ``hat guessing games''
are originally suggested 
by Gardner 
\cite{Gar:mathpuzzle}, 
and have been studied 
in connection with  
combinatorics and set theory. 
Set-theoretic studies before 2013 
of hat puzzles 
are summarized by Hardin and Taylor \cite{HT:coord}. 

In the present paper 
we adopt 
the formalization of hat puzzles 
in the Hardin--Tayler book \cite{HT:coord}. 
We call the players \emph{agents} instead of ``prisoners'' 
and the set of agents is often denoted by $A$. 
The set of hat colors is often denoted by $K$.  
We call  
a function $f$ from $A$ to $K$ 
a \emph{coloring}. 
%
A \emph{visibility graph} is 
a directed graph $V$ on 
the set $A$ of vertices without loops, 
that is, 
$V\subseteq A\times A\ssm\{\langle x,x\rangle\st x\in A\}$. 
By a directed edge $\langle a,b\rangle\in V$ 
we intend to mean 
that the agent $a$ can see (the hat worn by) the agent $b$.  
For 
an agent 
$a\in A$, 
we 
write 
$V(a)=\{x\in A\st 
	\langle a,x\rangle\in V
	\}$.  
A 
visibility graph $V$ is \emph{complete} 
if 
$V=A\times A\ssm\{\langle x,x\rangle\st x\in A\}$
(and hence $V(a)=A\ssm\{a\}$ for all $a\in A$). 
In some cases, 
declarations 
by 
the 
agents are not simultaneous, 
and some agent $a$ can hear 
some of other agents' 
declarations 
before 
declaring his own guess. 
We write $H(a)$ to denote the set of agents 
who declare their guesses before 
the agent $a$ does. 

Basically, 
a \emph{strategy} 
for 
an agent $a\in A$ 
is a function $G_a$ from ${^{V(a)\disju H(a)}K}$ to $K$, 
where $X\disju Y$ means the union when we regard $X$ and $Y$ 
as their disjoint copies. 
This reflects the manner of each agent $a$, 
who 
guesses the color of his own hat 
from the colors of hats 
worn by agents 
in 
his vision 
$V(a)$ 
and the declarations of 
agents 
in 
his audition 
$H(a)$. 
%
We call 
a 
collection 
$P=\{\langle a,G_a\rangle\st a\in A\}$ 
of all pairs of an agent and his strategy 
a \emph{predictor}. 

During the studies of hat puzzles, 
various protocols of 
guessing 
are suggested. 
Here 
we introduce several protocols  
and related results.





First, 
we consider 
a hat puzzle with 
complete visibility, 
where 
all agents declare their guesses 
at once. 
In other words, 
$V(a)=A\ssm\{a\}$ 
and $H(a)=\emptyset$ for all $a\in A$.  
Using the axiom of choice, 
%
%
in the case of $|K|=2$, 
Lenstra proved 
a 
significant result 
%
(see \cite[Theorem~3.3.1]{HT:coord} 
for a proof). 
\begin{thm}[Lenstra]\label{thm:Le}
Consider a hat puzzle 
with an arbitrary set $A$ of agents, 
the set $K=\{0,1\}$ of colors, 
complete visibility 
and simultaneous declaration. 
Then there is a predictor 
under which 
all agents guess correctly 
or 
all agents guess incorrectly. 
\end{thm}

Second, 
we consider a hat puzzle with 
the set $A$ of agents 
and 
complete visibility, where 
a designated agent $s$
declares first 
and then others do at once. 
We will call the agent $s$ the \emph{signaler}, 
and his declaration a \emph{signal}. 
Note that $V(s)=A\ssm\{s\}$ and $H(s)=\emptyset$, 
and 
for all $a\in A\ssm\{s\}$, 
$V(a)=A\ssm\{a\}$ and $H(a)=\{s\}$. 
%
Hardin and Taylor showed, 
under ZF${}+{}$DC, 
that
there is a 
\emph{signaling predictor} 
if and only if 
there is a 
\emph{Lenstra predictor} 
(See \cite[Theorem~3.3.2]{HT:coord}). 

\begin{thm}[ZF${}+{}$DC]\label{thm:siglen}
With an arbitrary set $A$ of agents 
and 
the set $K=\{0,1\}$ of hat colors, 
the following conditions are equivalent: 
\begin{enumerate}
\item The visibility graph $V$ is complete, 
	and in the puzzle of 
	one-in-advance 
	protocol, 
	there is a predictor 
	under which 
	all agents but the signaler 
	guess correctly. 
\item In the puzzle of simultaneous declaration, 
	there is a predictor 
	under which 
	all agents guess correctly 
	or 
	all agents guess incorrectly. 
\end{enumerate}
\end{thm}

Third, 
we consider the following setting: 
The set $A$ of agents is well-oredered by the relation $<_{A}$, 
and the visibility graph $V$ is \emph{one-way forward-complete}, 
that is, 
$V=\{\langle x,y\rangle\in A\times A\st x\mathbin{<_{A}}y\}$ 
(and hence $V(a)=\{x\in A\st a\mathbin{<_{A}}x\}$). 
The agents declare their guesses one by one 
along the order $<_{A}$. 
Note that, for each $a\in A$, 
$H(a)=\{x\in A\st x\mathbin{<_{A}}a\}$. 
Geschke, Lubarsky and Rahn~\cite{GLR:hat} 
studied such a protocol of hat puzzles  
with 
$A=\omega$ 
and 
$K
	=\{0,1\}$, 
and obtained the following result. 

Throughout 
the present 
paper, 
for a function $\varphi$ with domain $D$, $d\in D$ 
and any value $v$, 
$\varphi[d|v]$ 
abbreviates the function 
$\big(\varphi\restrictedto(D\ssm\{d\})\big)
	\cup
	\{\langle d,v\rangle\}$. 

\begin{defn}\label{defn:parity}
A function 
$p$ 
from ${{}^\omega}2$ to $2$ 
is called 
a \emph{parity function} 
if 
for any $f\in {{}^\omega}2$  
and $n\in\omega$, 
$p(f[n|0]) = 1 - p(f[n|1])$ 
holds. 
\end{defn}

\begin{thm}[ZF%
]\label{thm:GLR}
The following conditions are equivalent: 
\begin{enumerate}
\item
	In the puzzle with the set $A=\omega$ of agents, 
	the set $K=\{0,1\}$ of colors, 
	one-way forward-complete visibility 
	and one-by-one declaration,  
	there is a predictor 
	ensuring 
	all agents 
	but the agent 0 guess correctly. 
\item 
	There is a parity function. 
\end{enumerate}
\end{thm}


It is not so hard to 
obtain 
a parity function 
under ZFC\@. 
Geschke, Lubarsky and Rahn 
pointed out 
that 
a parity function 
may not exist 
under ZF${}+{}$DC \cite[Theorem~10]{GLR:hat}. 
Theorem~\ref{thm:GLR}
tells us 
that the existence of 
such a predictor 
can be 
regarded as 
a weak choice principle 
(see \cite[Section~4]{GLR:hat} for details).

Theorems~\ref{thm:Le}, \ref{thm:siglen} and \ref{thm:GLR} 
are all stated in the context of 
hat puzzles with two hat colors. 
In the 
following sections, 
we will integrate these theorems 
and generalize them to puzzles with 
more than two 
colors, 
assuming that the set $K$ of colors 
is equipped with 
a structure of 
a 
commutative group. 
In particular, 
by letting $K$ be the cyclic group of order 2  
(that is, $K=\mathbb{Z}_2=(\{0,1\},+)$), 
we 
obtain 
the original theorems 
as corollaries of our theorem.

\section{Main result}

Throughout this section, 
we consider hat puzzles 
with 
the set $K$ of hat colors equipped with 
additive group operation $+$, 
and 
the set $\omega$ of agents. 
We vary visibility graphs 
and guessing protocols. 

All the argument in this section 
can be done under ZF${}+{}$DC. 

\begin{defn}
We consider a hat puzzle with 
complete visibility, where 
all agents declare their guesses at once. 
In this puzzle 
we describe 
a strategy $G_a$ for each agent $a\in \omega$ 
by a function 
from ${^\omega}K$ to $K$ 
such that 
the value $G_a(f)$ does not depend on the value of $f$ at $a$.  
A predictor $P$ in such a puzzle 
is called \emph{biased} 
if, 
for every coloring $f\in {{}^\omega}K$ 
there is $k\in K$ 
such that $P_n(f)=f(n)+k$ holds for all $n\in\omega$. 
\end{defn}

Lenstra's theorem (Theorem~\ref{thm:Le}) 
tells us that,  
in such a hat puzzle with $K=\{0,1\}$, 
there 
is 
a predictor 
under which everyone's guess is correct or 
everyone's guess is incorrect. 
We call such a predictor a \emph{Lenstra predictor}. 
Note that, 
when 
$K=\mathbb{Z}_2$, 
a predictor $P$ is a biased predictor 
if and only if $P$ is a Lenstra predictor. 
So we may regard a biased predictor 
as a generalization of a Lenstra predictor 
to the case when $|K|>2$.

\begin{defn}
We consider a hat puzzle with 
complete visibility, where 
a designated 
signaler
$s\in\omega$ 
declares first 
and then others do at once. 
We describe a 
strategy 
$G_a$ for each agent 
$a\in \omega$ 
in such a puzzle 
in the following way: 
$G_s$ is a function from ${^\omega}K$ to $K$ 
such that 
$G_s(f)$ does not depend on the value of $f$ at $s$, 
and for $a\in \omega\ssm\{s\}$, 
$G_a$ is a function from ${^\omega}K\times K$ to $K$ 
such that 
$G_a(f,k)$ does not depend on the value of $f$ at $a$ 
(the second 
argument $k$ 
is intended to receive 
a 
signal). 
%
A predictor 
$P=\{\langle n,G_n\rangle\st n\in\omega\}$ 
in such a puzzle 
is called \emph{signal-biased} 
if 
the following two conditions hold:
\begin{enumerate} 
\item 
	For 
	every coloring $f\in {{}^\omega}K$, 
	$P_n(f,P_s(f))=f(n)$ holds for every $n\in\omega\ssm\{s\}$ 
	(everyone but the signaler guesses correctly).  
\item For every coloring $f\in {{}^\omega}K$ 
	and any two colors $k,l\in K$, 
\[
	P_n(f,k)
		-
	P_n(f,l)
	=l-k
\]
	holds for  
	every $n\in\omega\ssm\{s\}$.
\end{enumerate}
\end{defn}

In the case of $K=\{0,1\}$, 
a predictor which satisfies the first clause 
in the above definition 
is called a \emph{signaling} predictor~\cite[Theorem~3.3.2]{HT:coord}. 
It is obvious that, when $K=\mathbb{Z}_2$, 
a signal-biased predictor is a signaling predictor. 

\begin{prop}
When $K=\mathbb{Z}_2$, 
a signaling predictor is 
a 
signal-biased predictor. 
\end{prop}

\begin{proof}
Let $s\in\omega$ be the designated signaler 
and $P=\{\langle n,G_n\rangle\st n\in\omega\}$ 
be a signaling predictor. 
It suffices to show that, 
for every $f\in{^\omega}2$ and every $n\in\omega\ssm\{s\}$, 
$G_n(f,0)\neq G_n(f,1)$ holds. 

Suppose not, 
and choose $f\in{^\omega}2$ and $n\in\omega\ssm\{s\}$ 
so that $G_n(f,0)=G_n(f,1)=i$. 
Let $f_0=f[n|0]$ and $f_1=f[n|1]$. 
Since the value of $G_n$ does not depend on $f(n)$, 
we have 
\[
	G_n(f_0,0)=G_n(f_1,0)
	\quad\text{ and }\quad
	G_n(f_0,1)=G_n(f_1,1). 	
\]
Since either $f_0$ or $f_1$ is identical to $f$, we have  
\[
	G_n(f_0,0)=G_n(f_1,0)	=	i
	=
	G_n(f_0,1)=G_n(f_1,1)	
\]
This means that, 
whichever the value $f(n)$ is, 
and whichever the signal is, 
the strategy $G_n$ suggests the same guess. 
In other words, 
$G_n$ has no way to switch the guess 
when 
the 
color of the hat worn by the agent $n$ is switched. 
This contradicts the assumption that $P$ is a signaling predictor, 
which says that 
for \emph{any} coloring $f$ 
the predictor $P$ tells the agent $n$ the correct guess.  
\end{proof}

\begin{defn}
For an additive group $(K,{+})$, 
a function 
$p$ 
from ${{}^\omega}K$ to $K$ 
is called 
a \emph{$K$-parity function} 
if 
for any $f\in {{}^\omega}K$, 
$n\in\omega$ and $k,l\in K$, 
$
p(f[n|k])-p(f[n|l])=
l-k
$ 
holds. 
\end{defn}

A parity function, 
which we defined in Definition~\ref{defn:parity}, 
is 
nothing other than a $\mathbb{Z}_2$-parity function.

\begin{defn}
We consider a hat puzzle with 
one-way forward-complete visibility, where 
agents declare their guesses one-by-one 
along the ordering on $\omega$. 
In such a puzzle, 
we describe a strategy $G_a$ of the agent $a$ 
by a function from ${^\omega}K$ to $K$ such that 
$G_a(f)$ does not depend on the value of $f$ at $a$, 
since 
$V(a)\cap H(a)=\emptyset$ and 
$V(a)\cup H(a)=\omega\ssm\{a\}$ for each agent $a\in \omega$. 
%
%
A predictor $P$ in such a puzzle 
is called \emph{starter-biased} 
if 
the following two conditions hold:
\begin{enumerate} 
\item 
	For every coloring $f\in {{}^\omega}K$, 
	$P_n\big(f[0|P_0(f)]\big)=f(n)$ 
	holds for every $n\in\omega\ssm\{0\}$ 
	(everyone but the agent $0$ guesses correctly).  
\item For every coloring $f\in {{}^\omega}K$ 
	and any two colors $k,l\in K$, 
\[
	P_n(f[0|k])
		-
	P_n(f[0|l])
	=l-k
\]
	holds for  
	every $n\in\omega\ssm\{0\}$.
\end{enumerate}
\end{defn}

\begin{thm}\label{thm:equiv}
For hat puzzles with the set $\omega$ of agents 
and a set $K$ of hat colors 
equipped with an additive group operation $+$, 
the following conditions are equivalent:
\begin{enumerate}
\item\label{item:biased}
	There is a biased predictor 
	in the puzzle with complete visibility 
	and at-once declaration. 
\item\label{item:signalbiased} 
	There is a signal-biased predictor 
	in the puzzle with complete visibility 
	and 
	one-in-advance 
	declaration. 
\item\label{item:starterbiased}
	There is a starter-biased predictor 
	in the puzzle with one-way forward-complete visibility 
	and one-by-one declaration. 
\item\label{item:parity} 
	There is a $K$-parity function. 
\end{enumerate}
\end{thm}

\begin{proof}
(\ref{item:biased})${}\to{}$(\ref{item:signalbiased}):
	Suppose that 
	$P^{\mathrm{b}}
		=\{\langle n,G_n^{\mathrm{b}}\rangle
		\st n\in\omega\}$ 
	is a biased predictor. 
	We shall construct a signal-biased predictor
	$P^{\mathrm{sb}}
		=\{\langle n,G_n^{\mathrm{sb}}\rangle
		\st n\in\omega\}$. 
	Let $s\in\omega$ be the designated signaler. 
	For each coloring $f\in{{}^\omega}K$, 
	we set 
		$G_s^{\mathrm{sb}}(f)
		=G_s^{\mathrm{b}}(f)$, 
	and for $n\in\omega\ssm \{s\}$ and $x\in K$,  
		$G_n^{\mathrm{sb}}(f,x)
		=G_n^{\mathrm{b}}(f)
			-(x-f(s))$. 
	Now we verify that this works. 
	It is easy to see that $P^{\mathrm{sb}}$ satisfies 
		the second clause of the definition 
		of a signal-biased predictor. 
	Now 
	fix $n\in\omega\ssm\{s\}$ 
	and we will show that the agent $n$ guesses correctly. 
	Since $P^{\mathrm{b}}$ is a biased predictor, 
	we have 
\[
		G_n^{\mathrm{b}}(f)-f(n)
			=G_s^{\mathrm{b}}(f)-f(s), 
\]
	and by the definition of $G_n^{\mathrm{sb}}$, 
	we have 
	\begin{align*}
		G_n^{\mathrm{sb}}(f,G_s^{\mathrm{sb}}(f))
		&=	G_n^{\mathrm{b}}(f)-(G_s^{\mathrm{b}}(f)-f(s))	\\
		&=	G_n^{\mathrm{b}}(f)-(G_n^{\mathrm{b}}(f)-f(n))	\\
		&=	f(n). 
	\end{align*}		
	
(\ref{item:signalbiased})${}\to{}$(\ref{item:biased}):
	Suppose that 
	$P^{\mathrm{sb}}
		=\{\langle n,G_n^{\mathrm{sb}}\rangle
		\st n\in\omega\}$ 
	is a signal-biased predictor 
	where $s$ is the signaler. 
	We shall construct a biased predictor 
	$P^{\mathrm{b}}
		=\{\langle n,G_n^{\mathrm{b}}\rangle
		\st n\in\omega\}$. 
	For each $f\in {^\omega}K$, let 
	$G_s^{\mathrm{b}}(f)
		=G_s^{\mathrm{sb}}(f)$ 
	and for $n\in\omega\ssm\{s\}$, 
	$G_n^{\mathrm{b}}(f)
		=G_n^{\mathrm{sb}}(f,f(s))$.  
	It suffices to show that 
	$G_n^{\mathrm{b}}(f)-f(n)
		=G_s^{\mathrm{b}}(f)-f(s)$ 
	for every $n\in\omega\ssm\{s\}$. 
	Since $P^{\mathrm{sb}}$ is a signal-biased predictor, 
	we have 
\[
		G_n^{\mathrm{sb}}(f,f(s))
			-G_n^{\mathrm{sb}}(f,G_s^{\mathrm{sb}}(f))
			=G_s^{\mathrm{sb}}(f)-f(s). 
\]
	By the assumption, 
	$G_n^{\mathrm{sb}}(f,G_s^{\mathrm{sb}}(f))=f(n)$ 
	holds. 
	By the definition of $G_n^{\mathrm{b}}$
	for $n\in\omega\ssm\{s\}$, we have 
	\begin{align*}
		G_n^{\mathrm{b}}(f)-f(n)
		&=	G_n^{\mathrm{sb}}(f,f(s))-f(n)	\\
		&=	G_n^{\mathrm{sb}}(f,G_s^{\mathrm{sb}}(f))
			+G_s^{\mathrm{sb}}(f)
			-f(s)-f(n)	\\
		&=	G_s^{\mathrm{sb}}(f)
			-f(s). 
	\end{align*}

(\ref{item:signalbiased})${}\to{}$(\ref{item:parity}):
	Suppose that 
	$P^{\mathrm{sb}}
		=\{\langle n,G_n^{\mathrm{sb}}\rangle
		\st n\in\omega\}$ 
	is a signal-biased predictor 
	where $s$ is the signaler. 
	We define $p:{^\omega}K\to K$ 
	by letting 
\[
		p(f)=
		G_s^{\mathrm{sb}}(f)-f(s)
\]
	for each $f\in{^{\omega}}K$.  
	We check that $p$ is a $K$-parity function. 
	Since $G_s^{\mathrm{sb}}(f)$
	does not depend on $f(s)$, 
	we have 
	$p(f[s|k])-p(f[s|l])
		=l-k
		$. 
	For $n\in\omega\ssm\{s\}$, 
	by the definition of $p$, we have 
\[	
	p(f[n|k])-p(f[n|l])
	=	G_s^{\mathrm{sb}}(f[n|k])
		-G_s^{\mathrm{sb}}(f[n|l]).
\]	
Now,  
	under the predictor $P^{\mathrm{sb}}$, 
	every agent $n\in\omega\ssm\{s\}$ 
	guesses correctly for any coloring $f$.  
	Therefore, 
\[
	G_n^{\mathrm{sb}}(f[n|k],G_s^{\mathrm{sb}}(f[n|k]))
	=k
	\;\text{ and }\;
	G_n^{\mathrm{sb}}(f[n|l],G_s^{\mathrm{sb}}(f[n|l]))
	=l
	.
\]
Since $P^{\mathrm{sb}}$
is signal-biased 
and 
the value of $f$ at $n$ 
does not affect the value of $G_n^{\mathrm{sb}}$, 
we have 
\[
	l-k
	=	G_s^{\mathrm{sb}}(f[n|k])-G_s^{\mathrm{sb}}(f[n|l])
	=	p(f[n|k])-p(f[n|l]). 
\]	

(\ref{item:parity})${}\to{}$(\ref{item:signalbiased}):
Suppose that 
$p:{^\omega}K\to K$ 
is a $K$-parity function. 
We define a predictor $P=\{\langle n,G_n\rangle\}$ 
for a one-in-advance protocol with the designated signaler $s\in\omega$. 
For each $f\in{^\omega}K$, 
let $G_s(f)=p(f)$, 
and for $n\in\omega\ssm\{s\}$ and $k\in K$  
define $G_n(f,k)$ as 
\[
	G_n(f,k)=i	\iff	p(f[n|i])=k.	
\]
Note that, 
by the definition of a parity function, 
$P$ 
is well-defined and 
satisfies the second clause 
of the definition of a signal-biased predictor. 
We will show that $G_n(f,G_s(f))=f(n)$ 
for each $n\in\omega\ssm\{s\}$. 
Let $k=G_s(f)=p(f)$. 
Since $f[n|f(n)]=f$ trivially holds, 
	we have $p(f[n|f(n)])=k$, 
and 
by 
the definition of $G_n$ 
	we have	$G_n(f,k)=f(n)$.    

(\ref{item:starterbiased})${}\to{}$(\ref{item:parity}):
Suppose that, 
in a puzzle of one-by-one 
protocol 
with the set $\{t\}\cup\omega$ of agents, 
where $t$ is 
the starter 
(that is, the $<$-least agent), 
$P=\{\langle n,G_n\rangle\st n\in\{t\}\cup\omega\}$ 
is a starter-biased predictor. 
We 
define a 
function $p$ 
from ${^\omega}K$ to $K$ 
by letting 
$p(f)=G_t(f)$ for each $f\in{^\omega}K$. 
To see that $p$ is a $K$-parity function, 
fix $n\in\omega$ and $k,l\in K$. 
By the definition of 
$G_t$, 
we have 
$p(f[n|k])
	=G_t(f[n|k])$ 
and
$
p(f[n|l])
	=G_t(f[n|l])$. 
Since the strategy $G_n$ 
makes a correct guess 
and 
the value of $G_n$ 
does not depend on 
the value at $n$ of the function in the argument, 
we have 
\[
G_n(\{\langle t,p(f[n|k])\rangle\}\cup f)=k 
\;\text{ and }\;
G_n(\{\langle t,p(f[n|l])\rangle\}\cup f)=l. 
\]
Since $P$ is a starter-biased predictor, 
$p(f[n|k])-p(f[n|l])
	=
	l-k$
holds. 

(\ref{item:parity})${}\to{}$(\ref{item:starterbiased}):
Suppose that 
$p:{^\omega}K\to K$ 
is a $K$-parity function. 
We work in a puzzle of one-by-one 
protocol 
with the set $\{t\}\cup\omega$ of agents, 
where $t$ is 
the starter 
(that is, the $<$-least agent). 
We define 
a predictor 
$P=\{\langle n,G_n\rangle\st n\in\{t\}\cup\omega\}$. 
For each $f\in{^\omega}K$, 
let $G_t(f)=p(f)$, 
and for each $n\in\omega$ 
define $G_n(\{\langle t,k\rangle\}\cup f)$ by 
\[
	G_n(\{\langle t,k\rangle\}\cup f)=i
	{{}\iff{}}
	p(f[n|i])=k. 
\]
By the definition of $G_n$ and the property of $p$, 
$P$ 
is well-defined and 
satisfies the second clause of 
the definition of a starter-biased predictor. 
We will show, by induction on $n$, 
	$G_n(\{\langle t,p(f)\rangle\}\cup f)=f(n)$
	holds for every $n\in\omega$. 
Fix $n\in\omega$ 
and assume that every agent before $n$ 
except for $t$ 
guesses correctly.  
Then 
the agent $n$ will declare 
$i=G_n(\{\langle t,p(f)\rangle\}\cup f)$. 
By the definition of $G_n$, 
we have $p(f[n|i])=p(f)$, 
and by the property of $p$, 
$i=f(n)$ holds. 
\end{proof}

\section{Conclusion}

Here we check the existence of a $K$-parity function 
under ZFC. 


\begin{prop}
For any additive group $(K,{+})$, 
there is a $K$-parity function. 
\end{prop}

\begin{proof}
Define an equivalence relation ${=}^*$ 
on ${^\omega}K$ 
as follows:  
\begin{center}
$f{=}^*g$ 
if and only if 
$\{n\in\omega\st f(n)\neq g(n)\}$ is finite. 
\end{center}
Fix a complete system $A\subseteq{^\omega}K$ of representatives 
of the relation ${=}^*$. 
Define a function $\varphi$ from ${^\omega}K$ to $K$ 
as follows: 
For $f\in{^\omega}K$, 
let $\tilde{f}$ be the unique member of $A$ 
with $f\mathbin{{=}^*}\tilde{f}$,  
and define $\varphi(f)$ by 
\[
	\varphi(f)
	=
	\sum_{n\in\omega}\big(\tilde{f}(n)-f(n)\big). 
\]
It is easy to see that $\varphi$ is well-defined 
and is a $K$-parity function. 
\end{proof}

Note that, 
as we mentioned in Section~\ref{sec:intro}, 
a $\mathbb{Z}_2$-parity function 
may not exist under ZF${}+{}$DC\@.  
So we regard any of the condition 
listed in Theorem~\ref{thm:equiv} 
as an identical weak choice principle.

In the case of $K=\mathbb{Z}_2$, 
Theorem~\ref{thm:equiv} leads the following result, 
which extends both 
Theorem~\ref{thm:siglen} 
and 
Theorem~\ref{thm:GLR}.  

\begin{cor}
For hat puzzles with the set $\omega$ of agents 
and the 
set $K=\{0,1\}$ 
of hat colors, 
the following conditions are equivalent:
\begin{enumerate}
\item
	There is a Lenstra predictor 
	in the puzzle with complete visibility 
	and 
	simultaneous 
	declaration. 
\item 
	There is a signaling predictor 
	in the puzzle with complete visibility 
	and 
	one-in-advance 
	declaration. 
\item
	In the puzzle with one-way forward-complete visibility 
	and one-by-one declaration along the order on $\omega$, 
	there is a predictor under which 
	everyone but the agent $0$ guesses correctly. 
\item
	There is a parity function. 
\end{enumerate}
\end{cor}

\bibliographystyle{plain}
\bibliography{kada}

\bigskip

\begin{quote}
	\textsc{Masaru KADA}\\
	Graduate School of Science, Osaka Prefecture University\\
	1--1 Gakuen-cho, Naka-ku, Sakai Osaka 599--8531 JAPAN\\
	\texttt{kada@mi.s.osakafu-u.ac.jp}
\end{quote}

\medskip

\begin{quote}
	\textsc{Souji SHIZUMA}\\
	Graduate School of Science, Osaka Prefecture University\\
	1--1 Gakuen-cho, Naka-ku, Sakai Osaka 599--8531 JAPAN\\
	\texttt{dd305001@edu.osakafu-u.ac.jp}
\end{quote}

\end{document}